\documentclass[final,leqno]{siamltex704}

\usepackage{amsmath}
\usepackage{amssymb}
\usepackage{subfig}
\usepackage{graphicx}
\usepackage{epstopdf}
\usepackage{graphics}
\usepackage{float}
\usepackage{comment}
\usepackage{xcolor}
\usepackage{booktabs,etoolbox}
\usepackage{tikz}
\usepackage{enumerate}
\usepackage{hyperref}

\newtheorem{example}{\bf Example}[section]
\newtheorem{algorithm}{Stabilizer free Weak Galerkin Algorithm}

\newtheorem{remark}{Remark}
\def\3bar{{|\hspace{-.02in}|\hspace{-.02in}|}}

\newcommand{\doublenorm}[1]{\ensuremath{|\!|#1|\!|}}

\def\3bar{{|\hspace{-.02in}|\hspace{-.02in}|}}

\title{The lowest-order stabilizer free Weak Galerkin Finite Element Method}
\author{
	Ahmed AL-Taweel\thanks{Department of Mathematics and Statistics, 
		University of Arkansas at Little Rock, Little Rock, AR, 72204 (email:\href{mailto:Ahmed Al-Taweel@ualr.edu}{asaltaweel@ualr.edu})
		}
\and Xiaoshen Wang\thanks{Department of Mathematics and Statistics, University of Arkansas at Little Rock, Little Rock, AR, 72204 (email:\href{mailto:Xiaoshen Wang@ualr.edu}{xxwang@ualr.edu})
}}
\begin{document}
\maketitle

\begin{abstract}
Recently, a new stabilizer free weak Galerkin method (SFWG) is proposed, which is easier to implement and more efficient. The main idea is that by letting $j\geq j_{0}$ for some $j_{0}$, where $j$ is the degree of the polynomials used to compute the weak gradients, then the stabilizer term in the regular weak Galerkin method is no longer needed. Later on in \cite{al2019note}, the optimal of such $j_{0}$ for certain type of finite element spaces was given. In this paper, we propose a new efficient SFWG scheme using the lowest possible orders of piecewise polynomials for triangular meshes in $2 D$ with the optimal order of convergence.
\end{abstract}
\pagestyle{myheadings}
\section{Introduction}
In this paper, we are concerned with the development of an SFWG finite element method using the following Poisson equation 
\begin{eqnarray}
-\Delta u &=& f \quad \mbox{ in }\Omega, \label{bvp}  \\
u&=&0 \quad \mbox{ on }\partial\Omega, \label{dbc}
\end{eqnarray}
as the model problem, where $\Omega$ is a polygonal domain in $\mathbb{R}^2$.
The variational formulation of the Poisson problem (\ref{bvp})-(\ref{dbc}) is to seek $u\in H_{0}^{1}(\Omega)$ such that
\begin{eqnarray}
(\nabla u,\nabla v)&=&(f,v),\qquad \forall v\in H_{0}^{1}(\Omega).\label{weak form}
\end{eqnarray}

The standard weak Glaerkin (WG) method for the problem (\ref{bvp})-(\ref{dbc}) seeks weak Galerkin finite element solution $u_{h}=\{u_{0},u_{b}\}$ such that  
\begin{eqnarray}
(\nabla_w u_{h},\nabla_w v)+s(u_{h},v)&=&(f,v),\label{WGg form}
\end{eqnarray}
for all $v=\{v_{0},v_{b}\}$ with $v_{b}=0\mbox{ on } \partial \Omega$, where $\nabla_w$ is the weak gradient operator and $s(u_{h},v)$ in (\ref{WGg form}) is a stabilizer term that ensures a sufficient weak continuity for the numerical approximation. The WG method has been developed and applied to different types of problems, including convection-diffusion equations \cite{LinYeZhangZhu2018,GaoWangMu}, Helmholtz equations \cite{MuWangYeZhao2014,WangZhang2017,DuZhang2017}, Stokes flow \cite{WangYe2016,WangWangZhaiZhang}, and biharmonic problems \cite{MuWangYeBiharmonic2014}. Recently, Al-Taweel and Wang in \cite{al2020}, proposed the lowest-order weak Galerkin finite element method for solving reaction-diffusion equations with singular perturbations in $2D$. One of major sources of the complexities of the WG methods and other discontinuous finite element methods is the stabilization term. 

A stabilizer free  weak Galerkin finite element method is proposed by Ye and Zhang in \cite{ye} as a new method for the solution of the Poisson equation on polytopal meshes in 2D or 3D, where $\left( P_{k}(T), P_{k}(e),[P_{j}(T)]^{2}\right) $ elements are used. It is shown that there is a $j_{0}>0$ so that the SFWG method converges with optimal order of convergence for any $j\geq j_{0}$. However, when $j$ is too large, the magnitude of the weak gradient  can be extremely large, causing numerical instability. In \cite{al2019note}, the optimal $j_{0}$ is given to improve the efficiency and avoid unnecessary numerical difficulties. In this setting, if $\left( P_{k}(T), P_{k}(e),[P_{j}(T)]^{2}\right) $ elements are used for a triangular mesh, $j_{0}=k+1$, where $k\geq1$. In this paper, we propose a scheme using $\left(P_{0}(T),P_{1}(e),[P_{1}(T)]^{2}\right)$ elements for triangular meshes with the optimal order of convergence, which is more efficient than using the regular WG method $(P_{0}(T),P_{1}(e),[P_{1}(T)]^{2})$ elements.

The goal of this paper is to develop the theoretical foundation for using the lowest-order SFWG scheme to solve the Poisson equation (\ref{bvp})-(\ref{dbc}) on a triangle mesh in $2D$. The rest of this paper is organized as follows. In Section~\ref{sect:notation}, the notations and finite element space are introduced. Section~\ref{sect:erroreq} is devoted to investigating the error equations and several other required inequalities. The main error estimate is studied in Section~\ref{sect:errorestimate}. In Section~\ref{sect:L2error}, we will derive the optimal order $L^2$ error estimates for the SFWG finite element method for solving the equations (\ref{bvp})-(\ref{dbc}). Several numerical tests are presented in Section~\ref{sect:num}. Conclusions and some future research plans are summarized in Section~\ref{sect:conclusion}.

\section{Notations}\label{sect:notation}
In this section, we shall introduce some notations, and definitions. 
Suppose $\mathcal{T}_h$ is a quasi uniform triangular partition of $\Omega$. For every element $T\in\mathcal{T}_h$, denote $h_T$ as its diameter and $h=\max_{T\in\mathcal{T}_h}h_T$. Let $\mathcal{E}_h$ be the set of all the edges in $\mathcal{T}_h.$ The weak Galerkin finite element space is defined as follows: 
\begin{eqnarray}
	V_h=\{(v_0,v_b):v_0\in P_{0}(T),\forall T\in\mathcal{T}_h,\text{ and }v_b\in P_{1}(e),\forall e\in\mathcal{E}_h\}.
\end{eqnarray}
In this instance, the component $v_{0}$ symbolizes the interior value of $v$, and the component $v_{b}$ symbolizes the edge value of $v$  on each $T$ and $e$, respectively. Let $V^{0}_{h}$ be the subspace of $V_{h}$ defined as:
\begin{eqnarray}
	V_{h}^{0}&=&\{v: v\in V_{h},v_{b}=0\mbox{ on }\partial\Omega\}.
\end{eqnarray}

For each element $T\in\mathcal{T}_h$, let $Q_0$ be the $L^2$-projection onto $P_{0}(T)$ and let $\Bbb{Q}_{h}$ be the $L^2$-projection onto $[P_{1}(T)]^{2}$. On each edge $e$, denote by $Q_b$ the $L^2$-projection operator onto $P_1(e)$. Combining $Q_0$ and $Q_b$, denote by $Q_h=\{Q_0,Q_b\}$ the $L^{2}$-projection operator onto $V_h.$

For any $v=\{v_{0},v_{b}\}\in V_{h}$, the weak gradient $\nabla_wv\in [P_{1}(T)]^{2}$ is defined on $T$ as the unique polynomial satisfying
\begin{eqnarray}
(\nabla_wv,\vec{q})_T&=&-(v_0,\nabla\cdot \vec{q})_T+\langle v_b,\vec{q}\cdot\vec{n}\rangle_{\partial T}, \quad \forall \vec{q}\in [P_{1}(T)]^2,\label{EQ:WeakGradient}
\end{eqnarray}
where $\vec{n}$ is the unit outward normal vector of $\partial T$.

For simplicity, we adopt the following notations,
\begin{equation}
(v,w)_{{\cal T}_h}=\sum_{T\in{\cal T}_h}(v,w)_{T}=\sum_{T\in{\cal T}_h}\int_{T}v w dx,\nonumber
\end{equation}
\begin{equation}
\left\langle v,w\right\rangle _{\partial{\cal T}_h}=\sum_{T\in{\cal T}_h}\left\langle v,w\right\rangle _{\partial T}=\sum_{T\in{\cal T}_h}\int_{\partial T}v w dx.\nonumber
\end{equation}
\begin{algorithm}\label{alg:Wg}A numerical solution for (\ref{bvp})-(\ref{dbc}) can be obtain by finding $u_{h}=\{v_{0},v_{b}\}\in V_{h}^{0}$, such that the following equation holds
\begin{eqnarray}
(\nabla_w u_{h},\nabla_w v)_{{\cal T}_{h}}=(f,v_{0}),\quad \forall v=\{v_{0},v_{b}\}\in V_{h}^{0}.\label{WG form}
\end{eqnarray}
\end{algorithm}

We define an energy norm $\3bar \cdot \3bar$ on $V_{h}$ as:
\begin{equation}
\3bar v \3bar^2=\sum_{T\in{\cal T}_h}(\nabla_wv,\nabla_wv)_T
.\label{def-norm}
\end{equation}
An $H^{1}$ semi norm on $V_{h}$ is defined as:
\begin{equation*}
\|v\|_{1,h}^{2}=\sum_{T\in {\cal T}_{h}}\left( \|\nabla v_{0}\|^{2}_{T}+h_{T}^{-1}\|v_{0}-v_{b}\|^{2}_{\partial T}\right) .
\end{equation*}
\begin{remark}
	$\nabla v_{0}$ in the above definition is simply a placeholder, since $\nabla v_{0}|_{T}=0 \mbox{ for all } T\in {{\cal T}_{h}}$.
\end{remark}
\begin{lemma}\label{lnew}
	There exists $C>0$ so that 
	\begin{eqnarray*}
	\3barv\3bar\leq C\|v\|_{1,h},\qquad\forall v\in v_{h}.
	\end{eqnarray*}
\begin{proof}
	For any $T\in{{\cal T}_{h}}$ and $\vec{q}\in[P_{1}(T)]^{2}$, it follows from integration by part, trace inequality, and inverse inequality that
	\begin{eqnarray*}
	(\nabla_w v,\vec{q})_{T}&=&(\nabla v,\vec{q})_{T}+\left\langle v_{0}-v_{b},\vec{q}\cdot\vec{n}\right\rangle _{\partial T}\\&\leq&\|\nabla v\|_{T}\|\vec{q}\;\|_{T}+Ch^{\frac{-1}{2}}_{T}\|v_{0}-v_{b}\|_{\partial T}\|\vec{q}\;\|_{T}.
	\end{eqnarray*}
Letting $\vec{q}=\nabla_w v$ yields the result.
\end{proof}

The following lemma is one of the major results of this paper.
\end{lemma}
\begin{lemma}\label{mainlemma}
	For any  $v\in V_{h}, \mbox{ if } \nabla_w v|_{T_{i}}\in [P_{k+1}(T_{i})]^{2}, \forall i=1,2,T_{1}\cup T_{2}=e_{1}$, then
	\begin{eqnarray}
		\|v_{0}^{(1)}-v_{0}^{(2)}\|_{e_{1}}^{2}\leq Ch_{T_{1}}\|\nabla_w v\|_{T_{1}\cup T_{2}},\label{las1}
	\end{eqnarray}
	where $v_{0}^{(i)}=v_{0}|_{T_{i}},i=1,2$.
\end{lemma}
\begin{proof}
	Without loss, we may assume that the vertices of $T_{2}$ are $(0,0), (1,0), \mbox{ and } (0,1)$,
	\begin{eqnarray*}
		e_{1}=\{(x,0)|0\leq x\leq 1\},
	\end{eqnarray*}
	and the other edge of $T_{1}$ is $(a_{1},b_{1})$, where $b_{1}<0$.
	Denote $v_{0}^{(i)}=v_{0}|_{T_{i}}, i=1,2$. 
	\begin{figure}[h!]
		\centering
		\includegraphics[width=0.9\textwidth]{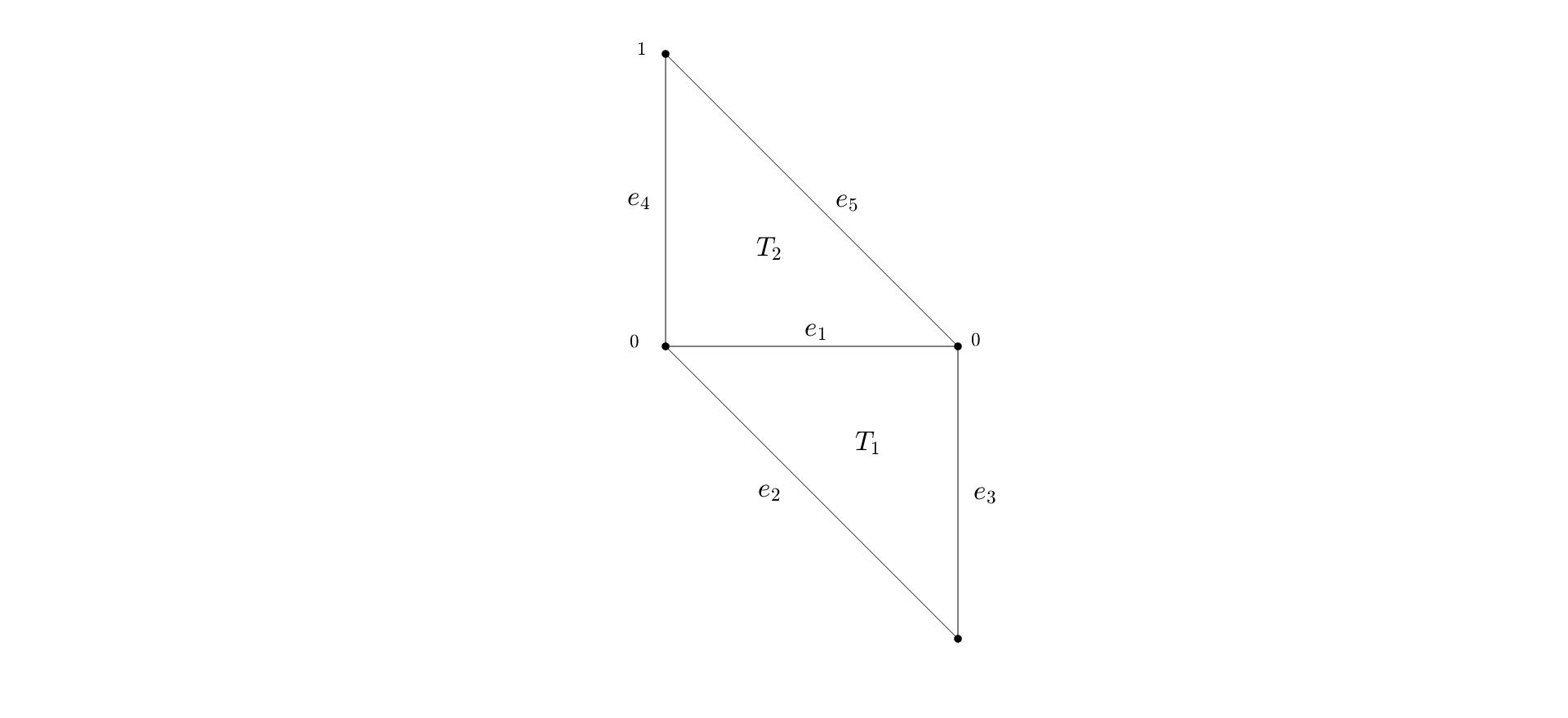}
		\caption{The two triangular mesh}\label{fi1}
	\end{figure}
Let $\vec{t}_{2}\mbox{ and } \vec{t}_{4}$ be unit tangents to $e_{2} \mbox{ and } e_{4}$, respectively; $L_{3}\mbox{ and } L_{5}$ be linear functions such that $L_{3}|_{e_{3}}=0 \mbox{ and } L_{5}|_{e_{5}}=0$. Let 
	\begin{eqnarray*}
		\vec{q}_{1}=\vec{q}\;|_{T_{1}}=L_{3}(x,y)(v_{0}^{(2)}-v_{0}^{(1)})\vec{t_{2}}
	\end{eqnarray*}
	and 
	\begin{eqnarray*}
		\vec{q}_{2}=\vec{q}\;|_{T_{2}}=L_{5}(x,y)(v_{0}^{(2)}-v_{0}^{(1)})\vec{t_{4}}.
	\end{eqnarray*} 
	Then
	\begin{eqnarray*}
		\vec{q}_{i}\cdot\vec{n}_{2i}=0,\quad \vec{q}_{i}|_{e_{2i+1}}=0, i=1,2.
	\end{eqnarray*}
	Scale $L_{1}\mbox{ and } L_{3}$ if necessary so that 
	\begin{eqnarray*}
		-L_{3}(0,0)\vec{t}_{2}\cdot\vec{n}_{1}^{(1)}=1=L_{5}(0,0)\vec{t}_{4}\cdot\vec{n}_{1}^{(2)},
	\end{eqnarray*}
	where $\vec{n}_{1}^{(i)}$ is the unit outwards normal vector of $e_{1}\in \partial T_{i},i=1,2$.\\ Since $L_{5}(1,0)\vec{t}_{2}\cdot\vec{n}_{1}^{(2)}=0$,
	\begin{eqnarray*}
		L_{5}(x,y)\vec{t}_{4}\cdot\vec{n}_{1}^{(2)}=\tilde{L}_{5}(x,y)=1-x-y.
	\end{eqnarray*}
	Similarly
	\begin{eqnarray*}
		L_{3}(x,y)\vec{t}_{2}\cdot\vec{n}_{1}^{(2)}=\hat{L}_{3}(x,y)=1-x+\alpha y, \mbox{ for some } \alpha.
	\end{eqnarray*} 
	It follows from the shape regularity assumptions that the slope of $e_{3}, \frac{1}{\alpha}$, satisfies $|\frac{1}{\alpha}|\geq \alpha_{0}>0$ for some $\alpha_{0}$. Since $\tilde{L_{3}}|_{e_{1}}=\tilde{L_{5}}|_{e_{1}}$,
	\begin{eqnarray*}
		(\nabla_w v,\vec{q})_{T_{1}\cup T_{2}}&=&\left\langle v_{b}-v_{0},\vec{q}\cdot\vec{n}\right\rangle _{\partial 
			T_{1}}+\left\langle v_{b}-v_{0},\vec{q}\cdot\vec{n}\right\rangle _{\partial T_{2}}\\&=&\left\langle v_{0}^{(2)}-v_{0}^{(1)}, (v_{0}^{(2)}-v_{0}^{(1)})\tilde{L_{3}}\right\rangle _{e_{1}}.
	\end{eqnarray*}
	Note that $0\leq \tilde{L_{5}}(x,y)\leq 1\mbox{ on } T_{1}\mbox{ and } 0\leq\tilde{L_{3}}(x,y)\leq 1\mbox{ on } T_{2}$. Then
	\begin{eqnarray*}
		(\nabla_w v, \vec{q}\;)_{T_{1}\cup T_{2}}&=&\left\langle v_{0}^{(2)}-v_{0}^{(1)}, (v_{0}^{(2)}-v_{0}^{(1)})\tilde{L_{3}}(x,y)\right\rangle _{e_{1}}\\&=&\int_{0}^{1}(v_{0}^{(2)}-v_{0}^{(1)})^{2}(1-x)d x\\&=&\frac{1}{2}\|v_{0}^{(2)}-v_{0}^{(1)}\|_{e_{1}}^{2}.
	\end{eqnarray*}
Thus
	\begin{eqnarray*}
		\|\vec{q}\;\|_{T_{2}}^{2}=\iint_{T_{2}}(1-x-y)^{2}(v_{0}^{(2)}-v_{0}^{(1)})^{2}d A=\frac{1}{12}(v_{0}^{(2)}-v_{0}^{(1)})^{2}\leq \|v_{0}^{(2)}-v_{0}^{(1)}\|_{e_{1}}^{2}.
	\end{eqnarray*}
	Similarly,
	\begin{eqnarray*}
		\|\vec{q}\;\|_{T_{1}}^{2}=\frac{1}{4}(1+\alpha+\frac{\alpha^{2}}{3})(v_{0}^{(2)}-v_{0}^{(1)})^{2}\leq\frac{1}{4}(1+\frac{1}{\alpha_{0}}+\frac{1}{3\alpha_{0}^{2}})\|v_{0}^{(2)}-v_{0}^{(1)}\|_{e_{1}}^{2}.
	\end{eqnarray*}
Thus
	\begin{eqnarray*}
		|(\nabla_w v,\vec{q}\;)_{T_{1}\cup T_{2}}|&\leq& \|\nabla_w v\|_{T_{1}}\|\vec{q}\;\|_{T_{1}}+\|\nabla_w v\|_{T_{2}}\|\vec{q}\;\|_{T_{2}}\\&\leq& C\|\nabla_w v\|_{T_{1}\cup T_{2}}\cdot \|v_{0}^{(2)}-v_{0}^{(1)}\|_{e_{1}}.
	\end{eqnarray*}
Thus after a scaling we have
	\begin{eqnarray*}
		 \|v_{0}^{(2)}-v_{0}^{(1)}\|_{e_{1}}\leq Ch_{T_{2}}^{\frac{1}{2}}\|\nabla_w v\|_{T_{1}\cup T_{2}}.
	\end{eqnarray*}
\end{proof}
\begin{lemma}
	Let $T_{1}\mbox{ and } T_{2}$ be such as in Lemma \ref{mainlemma}. Then 
	\begin{eqnarray}
	\|v_{b}-v_{0}\|^{2}_{\partial T_{1}\cup \partial T_{2}}\leq Ch_{T_{1}}\|\nabla_w v\|^{2}_{T_{1}\cup T_{2}}.
	\end{eqnarray}
\end{lemma}
\begin{proof}
	\begin{figure}[h!]
		\centering
		\includegraphics[width=0.9\textwidth]{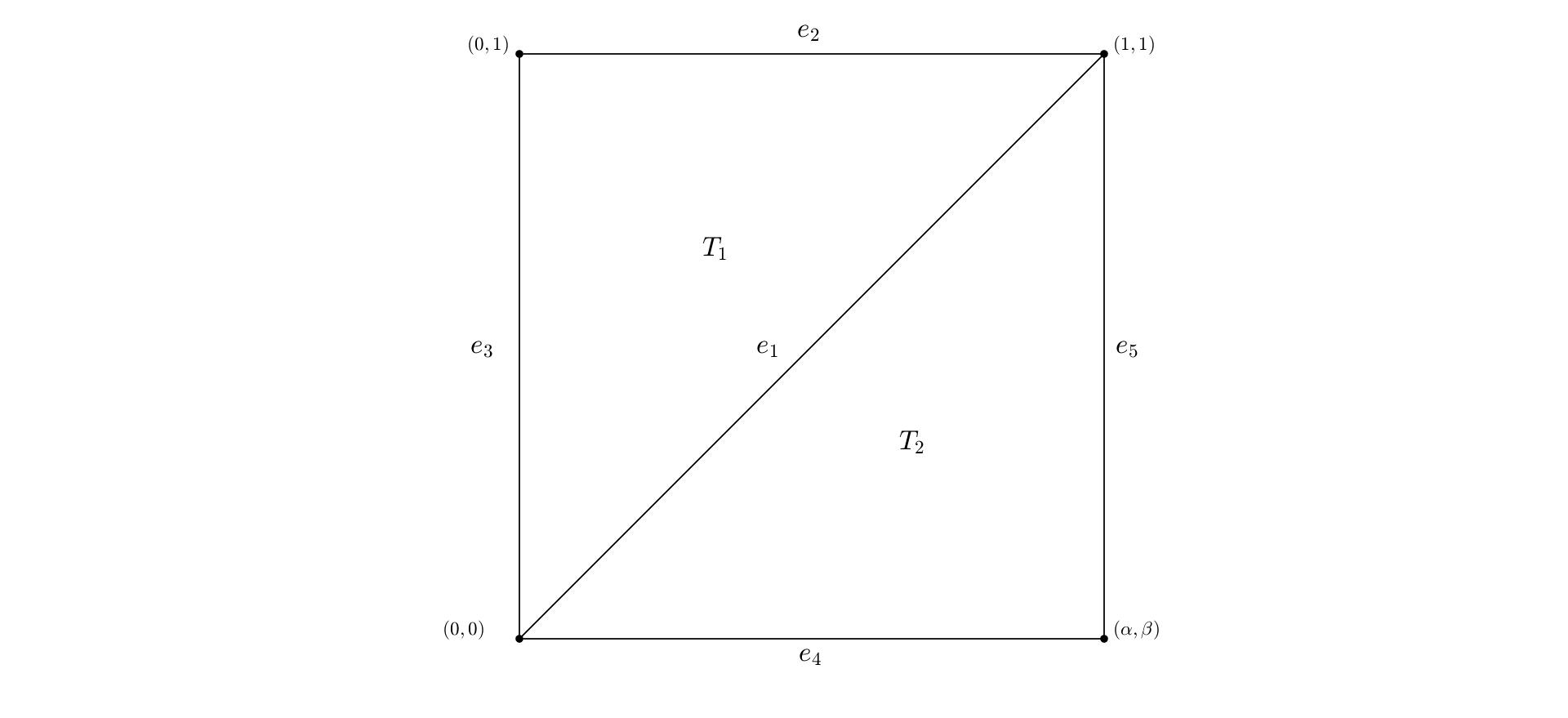}
		\caption{The two triangular mesh}\label{fi12}
	\end{figure}
	Without loss, we may assume that $T_{1}$ and $T_{2}$ are as shown in the Figure \ref{fi12}, where $e_{1}\cup e_{2}\cup e_{3}=\partial T_{1}, e_{1}\cup e_{4}\cup e_{5}=\partial T_{2}\mbox{ and } e_{1}=\partial T_{1}\cap \partial T_{2}$. If follows from Lemma \ref{mainlemma} that 
	\begin{eqnarray*}
		\|v_{0}^{(1)}-v_{0}^{(2)}\|^{2}_{e_{1}}\leq Ch_{T_{1}}\|\nabla_w v\|^{2}_{T_{1}\cup T_{2}}.
	\end{eqnarray*}
	We want to show 
	\begin{eqnarray}
	\|v_{b}-v_{0}\|^{2}_{e_{4}}\leq Ch_{T_{1}}\|\nabla_w v\|^{2}_{T_{1}\cup T_{2}}\label{eqnew1}
	\end{eqnarray}
	first.
	Let $L_{2}(x,y)=1-y$, then $L_{2}=0$ on $e_{2}$. Denote by $\vec{n}_{j}^{(i)}$ the unite outer ward normal vector to $\partial T_{i} \cap e_{j}, i=1,2, j=1,\dots,5$. Let $\vec{t}_{3}$ and $\vec{t}_{5}$ be unit tangent vectors to $e_{3}$ and $e_{5}$, respectively. Let 
	\begin{eqnarray*}
		\vec{q}\;|_{T_{1}}=\vec{q}_{1}=Q^{(1)}\vec{t}_{3},
	\end{eqnarray*}
	where $Q^{(1)}=a_{2}L_{2}$. Let $a$ be such that $a\vec{t}_{5}\cdot\vec{n}_{1}^{(2)}=-\vec{t}_{3}\cdot\vec{n}_{1}^{(1)}$ and 
	\begin{eqnarray*}
		\vec{q}\;|_{T_{2}}=\vec{q}_{2}=Q^{(2)}\vec{t}_{5},
	\end{eqnarray*}
	where $Q^{(2)}(x,y)=a\left(a_{2}(1-x)+(x-y)b_{2}\right)$. Then  $\vec{q}_{1}\cdot\vec{n}_{1}^{(1)}=-\vec{q}_{2}\cdot\vec{n}_{1}^{(2)}\mbox{ on } e_{1}$. Thus 
	\begin{eqnarray*}
	&&\left| \left\langle v_{b}-v_{0}^{(1)},\vec{q}_{1}\cdot\vec{n}^{(1)}\right\rangle _{e_{1}}+\left\langle v_{b}-v_{0}^{(2)},\vec{q}_{2}\cdot\vec{n}^{(2)}\right\rangle _{e_{1}}\right| \\&=&\left| \left\langle v_{0}^{(2)}-v_{0}^{(1)},\vec{q}_{1}\cdot\vec{n}^{(2)}\right\rangle _{e_{1}}\right|\\&\leq&\|v_{0}^{(2)}-v_{0}^{(1)}\|_{e_{1}}\|\vec{q}_{1}\|_{e_{1}}\\&\leq& Ch_{T_{1}}\|\nabla_w v\|_{T_{1}\cup T_{2}}\|\vec{q}_{1}\|_{e_{1}}\\&\leq& Ch_{T_{1}}\|\nabla_w v\|_{T_{1}\cup T_{2}}\|\vec{q}_{2}\|_{e_{1}}
	\end{eqnarray*} 
for any choice of $a_{2}$ and $b_{2}$. Write 
\begin{eqnarray*}
(v_{b}-v_{0})|_{e_{4}}=c_{0}+c_{1}x.
\end{eqnarray*}
Without loss, we assume $\vec{t}_{4}\cdot\vec{n}_{4}=1$.
 Let $aa_{2}=c_{0}, a(-a_{2}+b_{2})=c_{1}$. Then
 \begin{eqnarray*}
 \vec{q}_{2}\cdot\vec{n}_{1}^{(2)}=(v_{b}-v_{0})|_{e_{4}}.
 \end{eqnarray*}
Since
\begin{eqnarray*}
(\nabla_w v,\vec{q}\;)_{T_{1}\cup T_{2}}&=&\left\langle v_{0}^{(2)}-v_{0}^{(1)},\vec{q}_{1}\cdot\vec{n}_{1}^{(2)}\right\rangle _{e_{1}}+\left\langle v_{b}-v_{0},\vec{q}_{2}\cdot\vec{n}_{4}\right\rangle _{e_{4}}\\&=&\left\langle v_{0}^{(2)}-v_{0}^{(1)},\vec{q}_{1}\cdot\vec{n}_{1}^{(2)}\right\rangle _{e_{1}}+\|v_{b}-v_{0}\|_{e_{4}}^{2},
\end{eqnarray*}
\begin{eqnarray*}
\|v_{b}-v_{0}\|_{e_{4}}^{2}\leq \|\nabla_w v\|_{T_{1}\cup T_{2}}\|\vec{q}\;\|_{T_{1}\cup T_{2}}+C\|\nabla_w v\|_{T_{1}\cup T_{2}}\|v_{b}-v_{0}\|_{e_{4}}.
\end{eqnarray*}
Note that 
\begin{eqnarray*}
\|v_{b}-v_{0}\|_{e_{4}}^{2}&=&\int_{1}^{\alpha}(c_{0}+c_{1}x)^{2}dx\\&=&\frac{\alpha}{3}\left(c_{0}^{2}\alpha^{2}+\alpha c_{0}c_{1}+c_{1}^{2}\right)\\&\geq&\frac{\alpha}{6}\left(1+\alpha^{2}-\sqrt{\alpha^{4}-\alpha^{2}+1}\right)\left(c_{0}^{2}+c_{1}^{2}\right)\\&\geq&\frac{\alpha^{3}}{6(1+\alpha^{2})}(c_{0}^{2}+c_{1}^{2}).  
\end{eqnarray*}
It is easy to see that
\begin{eqnarray*}
\|\vec{q}\;\|_{T_{1}}^{2}=\frac{1}{a^{2}}\int_{T_{1}}c_{0}^{2}(1-y)^{2}dA=\frac{c_{0}^{2}}{12a^{2}},
\end{eqnarray*}
where $a=-\vec{t}_{5}\cdot\vec{n}_{1}^{(1)}/(\vec{t}_{5}\cdot\vec{n}_{1}^{(2)})$. It can be shown that
\begin{eqnarray*}
\|\vec{q}\;\|_{T_{2}}^{2}&=&\int_{T_{2}}\left(c_{0}(1-x)+(x-y)(c_{0}+c_{1})\right)^{2}dA\\&=&\frac{1}{12(a-b)^{2}}\left(c_{0}^{2}(3a^{2}-3ab+b^{2})+(3a-b)c_{1}c_{0}+c_{1}^{2}\right)\\&\leq&\frac{1}{12(a-b)^{2}}\left(4(a^{2}+b^{2})+1\right)\left(c_{0}^{2}+c_{1}^{2}\right) .  
\end{eqnarray*}
It follows from the shape regularity conditions that 
\begin{eqnarray*}
a&\geq&\alpha_{0},\\(a-b)&\geq&\alpha_{0},\\a^{2}+b^{2}&\leq& \beta_{0},
\end{eqnarray*}
for some $\alpha_{0}>0$ and $\beta_{0}>0$. Thus
\begin{eqnarray*}
\|\vec{q}\;\|^{2}_{T_{1}\cup T_{2}}\leq C\|v_{b}-v_{0}\|_{4}^{2}
\end{eqnarray*}
for some $C$. Thus
\begin{eqnarray*}
\|v_{b}-v_{0}\|_{e_{4}}\leq C\|\nabla_w v\|_{T_{1}\cup T_{2}}.
\end{eqnarray*}
Using a scaling argument, we have 
\begin{eqnarray*}
\|v_{b}-v_{0}\|_{e_{4}}\leq Ch_{T_{2}}\|\nabla_w v\|_{T_{1}\cup T_{2}}.
\end{eqnarray*}
Similarly, we can show that
\begin{eqnarray*}
	\|v_{b}-v_{0}\|_{e_{i}}^{2}\leq Ch_{T_{1}}\|\nabla_w v\|_{T_{1}\cup T_{2}}^{2},\; i=2,3,5.
\end{eqnarray*}
Now let's look at $\|v_{b}-v_{0}^{(1)}\|_{e_{1}\in\partial T_{2}}$. Let $\vec{q}_{2}=(a+bx)\vec{t}_{5}$, where $\vec{t}_{5}\cdot\vec{n}_{1}^{(2)}=\frac{1}{\sqrt{2}}$. Then
\begin{eqnarray*}
(\nabla_w v,\vec{q}_{2})_{T_{2}}=\left\langle v_{b}-v_{0},\frac{1}{\sqrt{2}}(a+bx)\right\rangle_{e_{1}}+\left\langle v_{b}-v_{0},a+bx\right\rangle_{e_{4}}, 
\end{eqnarray*}
choose $a$ and $b$ so that 
\begin{eqnarray*}
\left\langle v_{b}-v_{0},\frac{1}{\sqrt{2}}(a+bx)\right\rangle_{e_{1}}=\|v_{b}-v_{0}\|_{e_{1}}^{2}.
\end{eqnarray*}
Using a similar argument as when we were deriving estimate for $\|v_{b}-v_{0}\|_{e_{4}}$, we can get 
\begin{eqnarray*}
\|v_{b}-v_{0}\|^{2}_{e_{1}\in\partial T_{2}}\leq Ch_{T_{2}}\|\nabla_w v\|^{2}_{T_{1}\cup T_{2}}.
\end{eqnarray*}
Thus
\begin{eqnarray*}
	\|v_{b}-v_{0}\|^{2}_{\partial T_{1}\cup \partial T_{2}}\leq Ch_{T_{1}}\|\nabla_w v\|^{2}_{T_{1}\cup T_{2}}.
\end{eqnarray*}
\end{proof}
\begin{corollary}\label{corollary}
\begin{eqnarray}
\sum_{T\in {\cal T}_{h}} h^{-1}_{T}\|v_{b}-v_{0}\|_{\partial T}^{2}\leq C\3bar v\3bar^{2}.
\end{eqnarray}
\end{corollary}

Combining Lemmas \ref{lnew} and Corollary \ref{corollary}, we have the following theorem.
\begin{theorem}\label{newthto}
	There exists $C_{1}>0 \mbox{ and } C_{2}>0$ such that 
\begin{eqnarray*}
C_{1}\| v\|_{1,h}\leq \3bar v\3bar\leq C_{2}\|v\|_{1,h},\qquad \forall v\in V_{h}^{0}.
\end{eqnarray*}
\end{theorem}
\section{Error equation}\label{sect:erroreq}
In this section, we derive the error estimate of Algorithm~\ref{alg:Wg}.

\begin{lemma}\label{le50} For any function $\psi\in H^1(T)$, the following trace inequality holds true:
\begin{eqnarray}
	\|\psi\|_e^2\le C\left(h_T^{-1}\|\psi\|_T^2+h_T\|\nabla\psi\|_T^2\right).\label{ieq:trace-1}
\end{eqnarray}
\end{lemma}
\begin{lemma}\label{le6}(Inverse Inequality) There exists a constants $C$ such that for any piecewise polynomial $\psi|_{T}\in P_{k}(T)$,
	\begin{align}
		\doublenorm{\nabla \psi }_{T}\leq Ch^{-1}_{T}\doublenorm{\psi}_{T}, \qquad \forall T\in {\cal T}_h.\label{eqle6}
	\end{align} 
\end{lemma}	
\begin{lemma}\label{lt}
	Let $u\in H^{2+i}_{0}(\Omega), i=0,1,$ be the solution of the problem and $\mathcal{T}_h$ be a finite element partition of $\Omega$ satisfying the shape regularity assumptions. Then, the $L^2$ projections $Q_0 \mbox{ and } \Bbb{Q}_{h}$ satisfies
	\begin{eqnarray}
		&&\sum_{T\in\mathcal{T}_h}\left(\|u-Q_0u\|_T^2+h_T^2\|\nabla(u-Q_0u)\|_T^2\right)\le Ch^{2}\|u\|_{1}^2\label{ieq:Q0},\\
		&&\sum_{T\in\mathcal{T}_h}\left(\|\nabla u-\mathbb{Q}_h\nabla u\|_T^2+h_T^2\|\nabla u-\mathbb{Q}_h\nabla u)\|_{1,T}^2\right)\le Ch^{2(1+i)}\|u\|_{2+i}^{2}, i=0,1.
	\end{eqnarray}
\end{lemma}
\begin{lemma}\label{fgv1}
	Let $\phi\in H^{1}(\Omega)$. Then for each element $T\in {\cal T}_{h}$, we have
	\begin{eqnarray}
		\Bbb{Q}_{h}(\nabla \phi)=\nabla_w Q_{h}\phi.\label{2.10l}
	\end{eqnarray}
	\begin{proof}
		By definition (\ref{EQ:WeakGradient}) and integration by parts, for each $\vec{q}\in [P_{1}(T)]^{2}$ we have 
		\begin{eqnarray*}
			(\Bbb{Q}_{h}(\nabla\phi),\vec{q})_{T}&=&-(\phi,\nabla\cdot \vec{q})_{T}+\left\langle \phi,\vec{q}\cdot\vec{n}\right\rangle _{\partial T}\\&=&-(Q_{0}\phi,\nabla\cdot \vec{q})_{T}+\left\langle Q_{b}\phi,\vec{q}\cdot\vec{n}\right\rangle _{\partial T}\\&=&(\nabla_w Q_{h}\phi,\vec{q})_{T}.\label{bgt}
		\end{eqnarray*}
		which implies (\ref{2.10l}).
	\end{proof}
\end{lemma}	
\begin{lemma}\label{le3}
	Let $\phi\in H^{1}(\Omega)$. Then for all $v\in V_{h}^{0}$, we have
	\begin{eqnarray}
		( \nabla \phi,\nabla v_{0})_{T}=( \nabla_w(Q_{h} \phi),\nabla_w v)_{T}&+&\left\langle (\Bbb{Q}_{h}(\nabla \phi)\cdot \vec{n},v_{0}-v_{b}\right\rangle _{\partial T}.\label{eq1le3}
	\end{eqnarray}
	\begin{proof}
		Let $\Bbb{Q}_{h}(\nabla\phi)=\vec{q}$ and $Q_{h}\phi=P$. By Lemma \ref{fgv1} $\vec{q}=\nabla_w P$.
		\begin{eqnarray}
			(\vec{q},\nabla v_{0})_{T}&=&-(\nabla\cdot \vec{q},v_{0})_{T}+\left\langle \vec{q}\cdot\vec{n},v_{0}\right\rangle _{\partial T},\label{2w2w}\\(\vec{q},\nabla_w v)_{T}&=&-(\nabla\cdot\vec{q},v_{0})_{T}+\left\langle \vec{q}\cdot\vec{n},v_{b}\right\rangle _{\partial T},\label{2w2w1}
		\end{eqnarray}
implies
		\begin{eqnarray*}
			(\vec{q},\nabla v_{0})_{T}&=&(\vec{q},\nabla_w v)_{T}+\left\langle \vec{q}\cdot\vec{n},v_{0}-v_{b}\right\rangle _{\partial T}\\&=&(\nabla_w P,\nabla_w v)_{T}+\left\langle \vec{q}\cdot\vec{n},v_{0}-v_{b}\right\rangle _{\partial T},
		\end{eqnarray*}
		which completes the proof.
	\end{proof}
\end{lemma}
\begin{lemma}
	Let $e_h=Q_hu-u_h\in V_h$. Then for any $v\in V_{h}^{0}$, we have
	\begin{eqnarray}
	(\nabla_w e_h,\nabla_w v)_{{\cal T}_{h}} = \ell(u,v),\label{eq:error-eq}
	\end{eqnarray}
	where $\ell(u,v)$ is defined as follows,
	\begin{eqnarray*}
	\ell(u,v) = \sum_{T\in\mathcal{T}_h}\langle(\nabla u-\Bbb{Q}_{h}\nabla u)\cdot\vec{n},v_0-v_b\rangle_{\partial T}.
	\end{eqnarray*}
	\end{lemma}
\begin{proof}
	Testing the equation (\ref{bvp}) by $v=\{v_0,v_b\}\in V_h$ and using the fact that $\sum_{T\in {\cal T}_{h}}\left\langle \nabla u \cdot\vec{n},v_{b}\right\rangle_{\partial T}=0$, we arrive at
	\begin{eqnarray}\label{eq1}
		(\nabla u,\nabla v_0)_{T\in {\cal T}_{h}}-\left\langle \nabla u\cdot\vec{n},v_{0}-v_{b}\right\rangle _{\partial T} = (f,v_0).
	\end{eqnarray}
It follows from Lemma \ref{le3} that
\begin{eqnarray}
( \nabla u,\nabla v_{0})_{T}=( \nabla_w(Q_{h} u),\nabla_w v)_{T}&+&\left\langle (\Bbb{Q}_{h}(\nabla u)\cdot \vec{n},v_{0}-v_{b}\right\rangle _{\partial T}.\label{eq2}
\end{eqnarray}
Combining (\ref{eq1}) and (\ref{eq2}) gives
\begin{eqnarray}\label{eq3}
(\nabla_w(Q_{h}u),\nabla_w v)_{{\cal T}_{h}}=(f,v_{0})+\ell(u,v).
\end{eqnarray}
Subtracting (\ref{WG form}) from the above equation yields the error equation (\ref{eq:error-eq}), and this completes the proof. 
\end{proof}

\section{Error Estimates}\label{sect:errorestimate}
We will derive error estimates in this section.


\begin{lemma}
	Let $u\in H_{0}^{2+i}(\Omega),i=0,1,$ be the solution of the problem (\ref{bvp})-(\ref{dbc}). Then for $v\in V_{h}^{0}$,
	\begin{eqnarray}
	\ell(u,v) \le Ch^{1+i}\|u\|_{2+i}\3bar v\3bar, i=0,1,\label{ieq:ell}
	\end{eqnarray}
	respectively.
\end{lemma}
\begin{proof}
	It follows from the definition of $Q_b, Q_0, \Bbb{Q}_{h}$, the Cauchy-Schwarz inequality, trace inequality (\ref{ieq:trace-1}), and Theron (\ref{newthto})
	\begin{eqnarray*}
		|\ell(u,v)|&\le& \sum_{T\in\mathcal{T}_h}\left|\langle(\nabla u-\Bbb{Q}_{h}\nabla u)\cdot\vec{n},v_0-v_b\rangle_{\partial T}\right|\\
		&=& C\sum_{T\in\mathcal{T}_h}\|\nabla u-\Bbb{Q}_{h}\nabla u\|_{\partial T} \|v_0-v_b\|_{\partial T}\\
		&\le& C\bigg(\sum_{T\in\mathcal{T}_h}h_{T}\|\nabla u-\Bbb{Q}_{h}\nabla u\|_{\partial T}^2\bigg)^{1/2}
		\bigg(\sum_{T\in\mathcal{T}_h}h_{T}^{-1}\|v_0-v_b\|_{\partial T}^2\bigg)^{1/2}\\&\leq &C h^{1+i}\|u\|_{2+i}\3bar v\3bar,i=0,1.
		\end{eqnarray*}
This completes the proof.
\end{proof}

\begin{theorem}\label{TH:WG}
	Let $u$ and $u_h\in V_h$ be the exact solution and SFWG finite element solution of the problem (\ref{bvp})-(\ref{dbc}) and (\ref{WG form}). In addition, assume the regularity of exact solution $u\in H_{0}^{2+i}(\Omega),i=0,1$, then there exists a constant $C$ such that
	\begin{eqnarray}
	\3bar Q_h u-u_h\3bar\le C h^{1+i}\|u\|_{2+i},i=0,1,\label{H1er}
	\end{eqnarray}
	respectively.
\end{theorem}

\begin{proof}
	It follows from error equation (\ref{eq:error-eq}) that
	\begin{eqnarray*}
		\3bar Q_hu-u_h\3bar^2 = (\nabla_w e_h,\nabla_w e_{h})_{{\cal T}_{h}} = \ell(u,e_h).
	\end{eqnarray*}
	Letting $v=e_h$ in (\ref{ieq:ell}), yields
	\begin{eqnarray*}
		\3bar Q_hu-u_h\3bar^2 \le C h^{1+i}\|u\|_{2+i}\3bar e_h\3bar,i=0,1,
	\end{eqnarray*}
	and this implies the conclusion.
\end{proof}
\section{Error Estimates in $L^{2}$ norm}\label{sect:L2error}
The duality argument is utilized to get $L^{2}$ error estimate. Let $e_{h}=\{e_{0},e_{b}\}=Q_{h}u-u_{h}$. The dual problem seeks $\Phi\in H_{0}^{2}(\Omega)$ satisfying
\begin{eqnarray}
-\Delta\Phi&=&e_{0}\label{dua01},\quad \mbox{ in }\Omega\\\Phi&=&0,\quad \mbox{ on } \partial \Omega\nonumber.
\end{eqnarray}
Suppose that the following $H^{2}$-regularity holds true
\begin{eqnarray}
\doublenorm{\Phi}_{2}\leq C\doublenorm{e_{0}}.\label{dual1}
\end{eqnarray}
\begin{theorem}
	Let $u_{h}=\{u_{0},u_{b}\}$ be the SFWG finite element solution of (\ref{WG form}). Assume that the exact solution $u\in H^{2+i}_{0}(\Omega),i=0,1$ and (\ref{dual1}) holds true. Then, there exists a constant $C$ such that 
	\begin{eqnarray}
	\doublenorm{Q_{0}u-u_{0}}&\leq& C h^{1+i}\|u\|_{2+i}, i=0,1,\label{l2erora}\\ \|u-u_{0}\|&\leq& Ch\|u\|_{2}\label{l2erora1}.
	\end{eqnarray} 
\end{theorem}
\begin{proof}
	Testing (\ref{dua01}) by $e_{0}$ and using the fact that $\sum_{K\in {\cal T}_{h}}\left\langle \nabla\Phi\cdot\vec{n},e_{b}\right\rangle _{\partial {\cal T}_{h}}=0$, we obtain
	\begin{eqnarray}
	\doublenorm{e_{0}}^{2}&=&(-\Delta\Phi,e_{0})\nonumber\\&=&(\nabla\Phi,\nabla e_{0})_{ {\cal T}_{h}}-\left\langle \nabla\Phi\cdot\vec{n},e_{0}-e_{b}\right\rangle_{\partial {\cal T}_{h}}.\label{l21}
	\end{eqnarray}
	Setting $\phi=\Phi$ and $v=e_{h}$ in (\ref{eq1le3}) yields
	\begin{eqnarray}
	( \nabla \Phi,\nabla e_{0})_{{\cal T}_{h}}=( \nabla_w(Q_{h} \Phi),\nabla_w e_{h})_{{\cal T}_{h}}&+&\left\langle \Bbb{Q}_{h}(\nabla \Phi)\cdot\vec{n},e_{0}-e_{b}\right\rangle _{\partial {\cal T}_{h}}.\label{l22}
	\end{eqnarray}
	Substituting (\ref{l22}) into (\ref{l21}) gives
	\begin{eqnarray}
	\doublenorm{e_{0}}^{2}&=&( \nabla_w(Q_{h} \Phi),\nabla_w e_{h})_{{\cal T}_{h}}+\ell(\Phi,e_{h}).\label{l2220}
	\end{eqnarray}
	Using equation \ref{WG form} and the error equation (\ref{eq:error-eq}), we have
	\begin{eqnarray}
	( \nabla_w(Q_{h} \Phi),\nabla_w e_{h})_{{\cal T}_{h}}&=&\ell(u,Q_{h}\Phi).\label{l245}
	\end{eqnarray} 
	By combining (\ref{l2220}) with (\ref{l245}), we obtain
	\begin{eqnarray}
	\doublenorm{e_{0}}^{2}&=&\ell(u,Q_{h}\Phi)+\ell(\Phi,e_{h}).\label{l2q2}
	\end{eqnarray}
	To bound the terms on the right-hand side of equation (\ref{l2q2}). We use the Cauchy-Schwarz inequality, the trace inequality (\ref{ieq:trace-1}) and the definition of $Q_{h}$ and $\Bbb{Q}_{h}$ to get 
	\begin{eqnarray}
	\left| \ell(u,Q_{h}\Phi)\right| &=&\left| \sum_{T\in{\cal T}_h }\left\langle \nabla u-\Bbb{Q}_{h}(\nabla u)\cdot \vec{n},Q_{0}\Phi-Q_{b}\Phi\right\rangle _{\partial T} \right| \nonumber\\ &\leq& \left( \sum_{T\in{\cal T}_h }\doublenorm{\nabla u-\Bbb{Q}_{h}\nabla u}^{2}_{\partial T}\right)^{\frac{1}{2}} \left( \sum_{T\in{\cal T}_h }\doublenorm{Q_{0}\Phi-Q_{b}\Phi}^{2}_{\partial T}\right)^{\frac{1}{2}}\nonumber \\ &\leq& C \left( \sum_{T\in{\cal T}_h }h_{T}\doublenorm{\nabla u-\Bbb{Q}_{h}\nabla u}^{2}_{\partial T}\right)^{\frac{1}{2}} \left( \sum_{T\in{\cal T}_h }h^{-1}_{T}\doublenorm{Q_{0}\Phi-\Phi}^{2}_{\partial T}\right)^{\frac{1}{2}}\nonumber\\ &\leq& Ch^{1+i}\|u\|_{2+i}\|\Phi\|_{1}\nonumber,
	\end{eqnarray}
	which implies
	\begin{eqnarray}
	\left| \ell(u,Q_{h}\Phi)\right|\leq Ch^{1+i}\|u\|_{2+i}\|\Phi\|_{2},i=0,1\label{l21e}.
	\end{eqnarray}
The estimates (\ref{ieq:ell}), (\ref{H1er}), and Lemma \ref{lt} give
	\begin{eqnarray}
	\left| \ell(\Phi,e_{h})\right| &=&\left| \left\langle \nabla \Phi-\Bbb{Q}_{h}(\nabla \Phi)\cdot\vec{n},e_{0}-e_{b}\right\rangle _{\partial {{\cal T}_{h}}} \right| \nonumber\\& \leq& C h^{1+i}\|u\|_{2+i}\|\Phi\|_{2},i=0,1.\label{l21e2}
	\end{eqnarray}
Now combining (\ref{l2q2}) with the estimates (\ref{l21e})-(\ref{l21e2}), we obtain 
	\begin{eqnarray}
	\doublenorm{e_{0}}^{2}\leq C h^{1+i}\|u\|_{2+i}\|\Phi\|_{2},i=0,1,
	\end{eqnarray}
	which combined with (\ref{dual1}) and the triangle inequality, provides the required error estimate (\ref{l2erora}). (\ref{l2erora1}) follows from 
	\begin{eqnarray*}
	\|u-u_{0}\|\leq \|u-Q_{0}u\|+\|Q_{0}u-u_{0}\|\leq Ch\|u\|_{2}.
	\end{eqnarray*}
\end{proof}
\section{Numerical Experiments}\label{sect:num}
In this section, two numerical examples in two dimensional uniform triangular meshes are presented to validate the theoretical results derived in previous sections. Since the regular SFWG method does not work for $\left(P_{0}(T), P_{1}(e), [P_{1}(T)]^{2}\right)$ elements, we'll compare our new SFWG method with the standard WG method.
\begin{example}
In this example, we use the SFWG scheme (\ref{WG form}) to solve the Poisson problem (\ref{bvp})-(\ref{dbc}) posed on the unit square $\Omega=(0,1)\times (0,1)$ with the analytic solution 
\begin{eqnarray}
u(x,y)=\sin(\pi x)\sin(\pi y).
\end{eqnarray}
The boundary conditions and the source term $f(x,y)$ are computed accordingly. The first two levels of meshes are plotted in Fig \ref{fi223}. Table \ref{tab:example1} shows errors and convergence rates in $H^{1}$-norm and $L^{2}$-norm comparison between the SFWG finite element method and the WG finite element method proposed in \cite{al2020}. As we can see in \ref{tab:example1} that the error between $u_{0}$, the numerical solution obtain from SFWG method, and $Q_{0}u$, the $L^{2}$-projection of $u$ is $\|Q_{0}u-u_{0}\|=O(h^{2})$. On the other hand, if the regular WG method is used, $\|Q_{0}u-u_{0}\|=O(h)$, much lower than $O(h^{2})$. Thus our new SFWG method is much more accurate.
\begin{figure}
	\centering
	\subfloat[Level 1]{{\includegraphics[width=5cm]{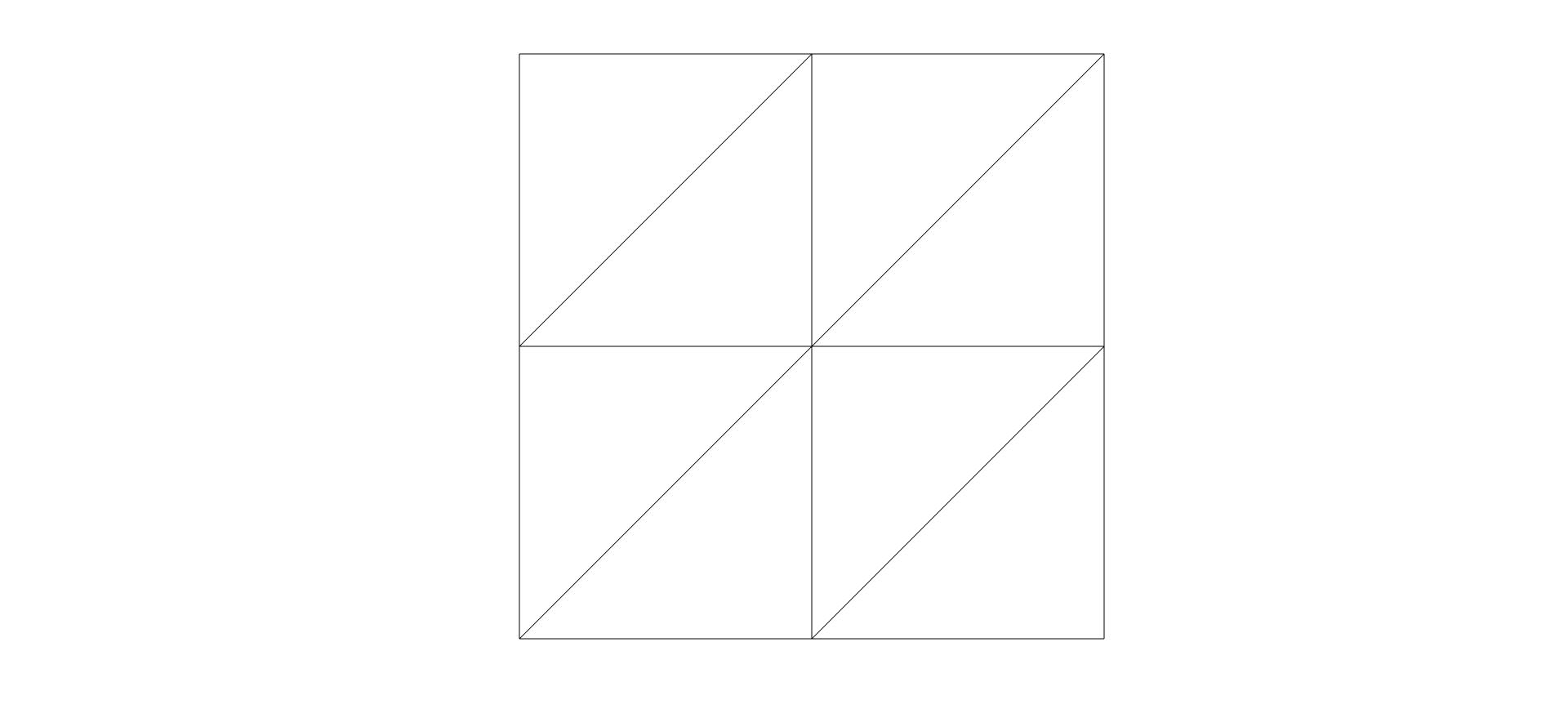} }}%
	\qquad
	\subfloat[Level 2]{{\includegraphics[width=5cm]{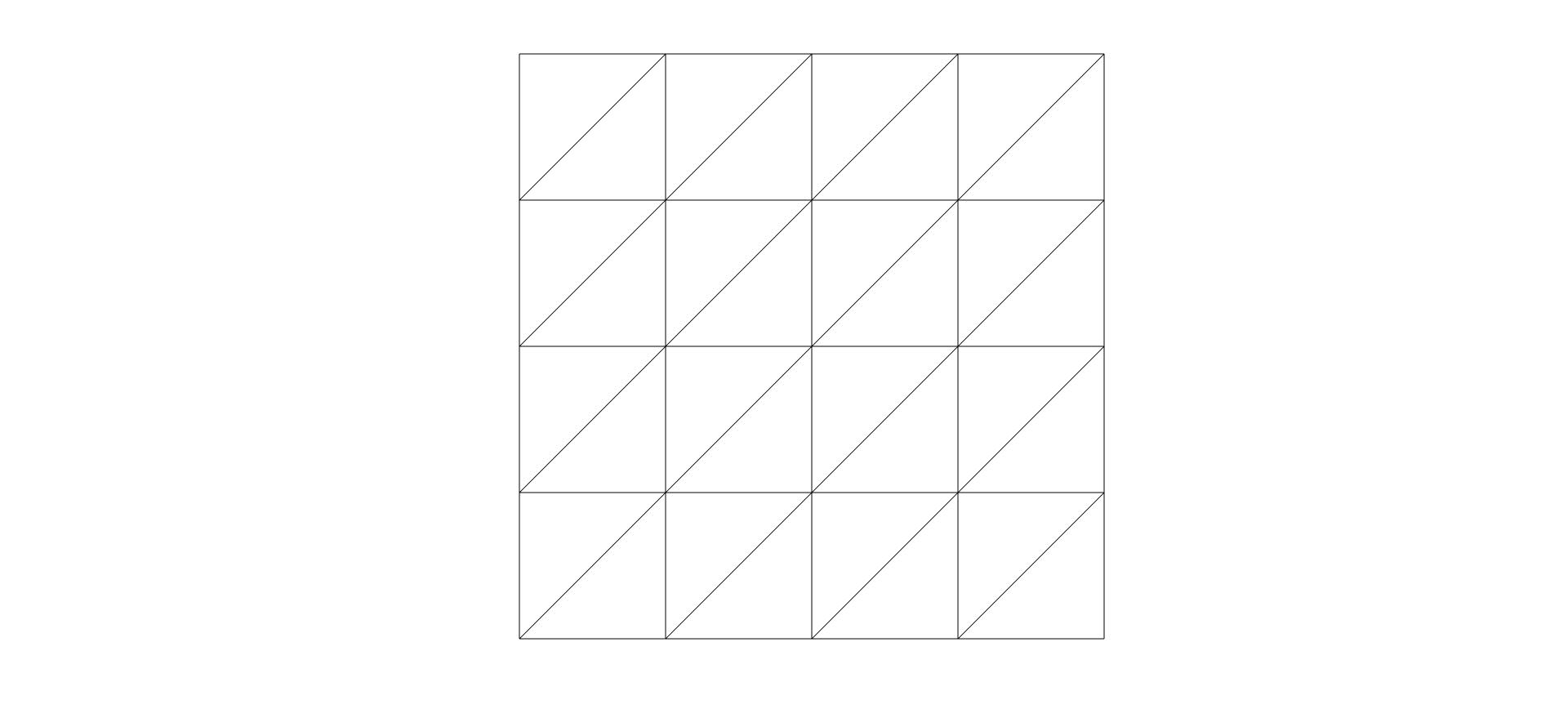} }}%
	\caption{A triangulation of a square domain in computation}%
	\label{fi223}%
\end{figure}%
\begin{table}[H]
	\caption{Errors and convergence results for $\left( P_{0}(T),P_{1}(e),[P_{1}(T)]^{2}\right) $ elements.} 
	\label{tab:example1}
	\small 
	\centering 
	\begin{tabular}{lccccccccr} 
		\toprule[\heavyrulewidth]\toprule[\heavyrulewidth]
		& &   \textbf{SFWG elements}& &  & &   \textbf{WG elements} &&  &  \\ 
		\midrule
		\textbf{$1/h$} & \textbf{$\3barQ_{h}u-u_{h}\3bar$} & \textbf{$\mbox{Rate}$} & \textbf{$\|Q_{0}u-u_{0}\|$} & \textbf{$\mbox{Rate}$}& \textbf{$\3barQ_{h}u-u_{h}\3bar$} & \textbf{$\mbox{Rate}$} & \textbf{$\|Q_{0}u-u_{0}\|$} & \textbf{$\mbox{Rate}$} \\ 
		\midrule
		2   &6.2075E-01 &- &8.8329E-02 &- & 1.1717E-00 &- &1.5192E-01  &-  \\
		4   & 1.8108E-01  & 1.78  &3.0651E-02  & 1.53 & 5.7311E-01 & 1.03 &9.3682E-02  &0.70\\
		8 & 4.7252E-02  & 1.94  &8.3544E-03  & 1.88 & 2.7088E-01 & 1.08 &4.8636E-02  &0.95 \\
		16  & 1.1952E-02  & 1.98  &2.1351E-03  & 1.97 & 1.2974E-01 & 1.06 &2.4264E-02 &1.00 \\  
		32  & 2.9971E-03  & 2.00  &5.3676E-04  & 1.99 & 6.3209E-02 & 1.04 &1.2045E-02  &1.00 \\ 
		64  & 7.5022E-04  & 2.00  &1.3438E-04  & 2.00 & 3.1159E-02 & 1.02 &5.9911E-03  &1.00 \\		
		\bottomrule[\heavyrulewidth] 
	\end{tabular}
\end{table}
Figure \ref{ti} shows the computational time (in seconds) comparison between SFWG finite element method and weak Galerkin finite element method. As we can see in Figure \ref{ti} that the SFWG algorithm is running faster than the standard weak Galerkin algorithm. We can also see in Figure \ref{ti} that the computation time with $8192$ elements by using the SFWG is $15.0469$, which is much less than $16.5156$, needed by using the standard weak Galerkin algorithm. Therefore, when a large number of elements are used the computation time becomes a significant factor. Thus the SFWG method is more efficient in both accuracy and computation time. Numerical example is carried out on a Laptop computer with 12.0 GB memory and Intel(R) Core (TM) i7-8550U CPU @ 1.80 GHz.
\begin{figure}[h!]
	\centering
	\includegraphics[width=0.9\textwidth]{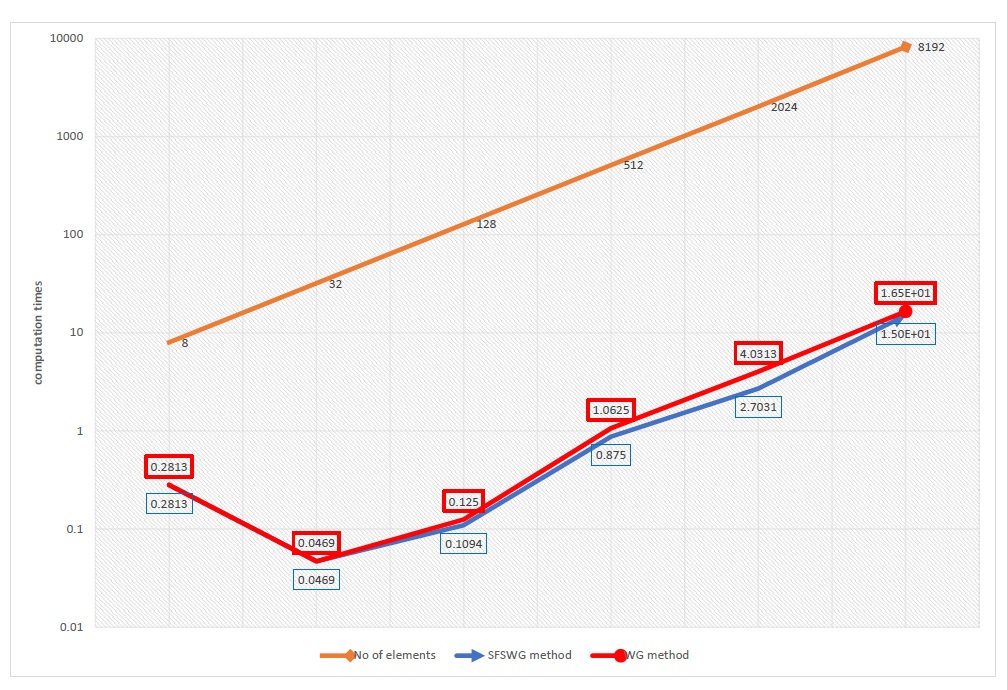}
	\caption{Comparison of computation times for $\left( P_{0}(T), P_{1}(e), [P_{1}(T)]^{2}\right)$ elements.}\label{ti}
\end{figure}
\end{example}
\begin{example}
	This example is adopted from \cite{soos}. Let $\Omega=(0, 1)^{2}$ and the boundary value condition (\ref{bvp}) is chosen such that the exact solution is
	\begin{align}\label{ex3}
	u(x,y)=r^{2/3}\sin\frac{2\theta}{3},\qquad (x,y)\in \Omega,
	\end{align} \label{singlularproblm}
	where polar coordinates $r=\sqrt{x^{2}+y^{2}}$ and $\theta=\arctan(\frac{y}{x})$ are used.
\begin{table}[H]
	\caption{Errors and convergence results for $\left( P_{0}(T),P_{1}(e),[P_{1}(T)]^{2}\right) $ elements.} 
	\label{tab:example 2}
	\small 
	\centering 
	\begin{tabular}{lccccccccr} 
		\toprule[\heavyrulewidth]\toprule[\heavyrulewidth]
		& &   \textbf{SFWG elements}& &  & &   \textbf{WG elements} &&  &  \\ 
		\midrule
		\textbf{$1/h$} & \textbf{$\3barQ_{h}u-u_{h}\3bar$} & \textbf{$\mbox{Rate}$} & \textbf{$\|Q_{0}u-u_{0}\|$} & \textbf{$\mbox{Rate}$}& \textbf{$\3barQ_{h}u-u_{h}\3bar$} & \textbf{$\mbox{Rate}$} & \textbf{$\|Q_{0}u-u_{0}\|$} & \textbf{$\mbox{Rate}$} \\ 
		\midrule
		2   &1.6754E-02  & -     &1.1548E-03  & - & 9.3655E-02 &- &3.5650E-03  &-  \\
		4   & 1.0645E-02  & 0.65 &3.7097E-04  & 1.64 & 5.6604E-02 & 1.03 &1.6153E-03  &1.14\\
		8 & 6.7121E-03 & 0.67   &1.1709E-04  & 1.66 & 3.1030E-02 & 1.08 &7.5599E-04  &1.10 \\
		16  & 4.2294E-03  & 0.67  &3.6893E-05  & 1.67 & 1.6352E-02 & 1.06 &3.7573E-04 &1.00 \\  
		32  & 2.6644E-03  & 0.67  &1.1621E-05  & 1.67 & 8.4922E-03 & 1.04 &1.8985E-04  &0.98 \\ 
		64  & 1.6784E-03  & 0.67  &3.6605E-06  & 1.67 & 4.4059E-03 & 1.02 &9.5830E-05  &0.99 \\		
		\bottomrule[\heavyrulewidth] 
	\end{tabular}
\end{table}	
	The exact solution $u\in H^{1+2/3}(\Omega)$. Table \ref{tab:example 2} shows the convergence rates in the $H^{1}$-norm is $O(h^{\frac{2}{3}})$ and $L^{2}$-norm is $O(h^{\frac{5}{3}})$ by using the SFWG algorithm. 
\end{example}
\section{Conclusion and Remark}\label{sect:conclusion}
In this paper, we have developed a new SFWG finite element methods for the Poisson equation (\ref{bvp})-(\ref{dbc}) on triangle mesh.
The stabilizer free setting has been used in the numerical scheme. The error estimate in energy norm has been provided and validated in the numerical tests.

Numerical experiments have shown that the new SFWG finite element method works for $\left( P_{k}(T), P_{k+1}(e), [P_{k+1}(T)]^{2}\right) $ elements in general. One of our future projects is to extend the theoretical results to the general cases.

\end{document}